\documentclass[twoside]{article}
\title{\bf FRAME MEASURES FOR INFINITELY MANY MEASURES}
\author{FARIBA ZEINAL ZADEH FARHADI$^{a}$, MOHAMMAD SADEGH ASGARI$^{b}$,\\ 
MOHAMMAD REZA MARDANBEIGI$^{a*}$}
\date{}

\usepackage{amsmath,amsthm,amssymb,amscd,amsfonts}
\usepackage[bookmarksnumbered, colorlinks, plainpages]{hyperref}
\theoremstyle{definition}
\newtheorem{theorem}{Theorem}[section]
\newtheorem{remark}[theorem]{Remark}

\numberwithin{equation}{section}

\newtheorem{lemma}[theorem]{Lemma}
\newtheorem{proposition}[theorem]{Proposition}
\newtheorem{corollary}[theorem]{Corollary}

\newtheorem{definition}[theorem]{Definition}
\newtheorem{example}[theorem]{Example}

\newcommand{\subject}[1]{\begin{flushleft}
\textbf{2010 AMS Subject Classification}: #1\end{flushleft}}
\newcommand{\keyword}[1]{\par\noindent \textbf{Keywords:} #1}

\newcommand{\eval}[2][\right]{\relax
\ifx#1\right\relax \left.\fi#2#1\rvert}

\renewcommand{\sectionmark}[1]{}
\begin{document}
\maketitle
\begin{abstract}
For every frame spectral measure $ \mu $, there exists a discrete measure $ \nu $ as a frame measure. Since if $ \mu $ is not a frame spectral measure, then there is not any general statement about the existence of frame measures $ \nu $ for $ \mu $, we were motivated to examine Bessel and frame measures. We construct infinitely many measures $ \mu $ which admit frame measures $ \nu $, and we show that there exist infinitely many frame spectral measures $ \mu $ such that besides having a discrete frame measure, they admit continuous frame measures too.
\noindent
\vspace{.3cm}
\keyword{Fourier frame, Plancherel theorem, spectral measure, frame measure, Bessel measure.}
\subject{Primary 28A80, 28A78, 42B05.}
\end{abstract}
\section{Introduction} 
Motivated by questions of fractal frame spectral measures, Bessel and frame measures were introduced in \cite{3}. In fact, frame measures are a generalization of Fourier frames. When $ L^2(\mu) $ has a Fourier frame, $ \mu $ is called a frame spectral measure and there exists a discrete measure $ \nu $ which is a frame measure for $ \mu $. So every frame spectral measure $ \mu $ admits a discrete frame measure $ \nu $. There has been a wide range of interest in identifying frame spectral measures especially, fractal ones. The interested reader can refer to \cite {2, 4, 5, 6, 7, 10, 11, 12, 13, 14, 15, 16, 17, 18}. If $ \mu $ is not a frame spectral measure, then there is not any general statement about the existence of frame measures for $ \mu $. Nevertheless, in \cite{3} the authors showed that if one frame measure $ \nu $ exists for $ \mu $, then one can obtain many frame measures for $ \mu $ by convolution of $ \nu $ and probability measures. 

In this paper we construct infinitely many measures $ \mu $ (by using convolutions of measures) which admit frame measures $ \nu $. In addition, we obtain that there exist infinitely many frame spectral measures such that besides having an associated discrete frame measure, they admit continuous frame measures too.

The rest of this paper is organized as follows: In Section 2 basic definitions and notation are given. Section 3 is devoted to identifying Bessel/frame measures $ \nu $ and constructing measures $ \mu $ which admit Bessel/frame measures $ \nu $. We show that a finite measure $ \nu $ is a Bessel measure for a finite measure $ \mu $, if and only if $ \mu $ is a Bessel measure for $ \nu $. Therefore, every finite measure $ \mu $ is a Bessel measure to itself (Corollary \ref{MK}). We investigate connections between the existence of a Bessel/frame measure for $ \mu $, $ \mu' $ and the sum $ \mu + \mu' $. If $ \mu $ is a Borel measure on $ \mathbb{R}^d $ and if $\nu $ is a Bessel/frame measure for $\mu $, then for any $ E\subset \mathrm{supp}\mu $, the measure $ \mu'=\chi_{E}d\mu $ admits $ \nu $ as a Bessel/frame measure with the same bound(s) (Corollary~\ref{Cco}). In Theorem \ref{3.3} we show that Lebesgue measure is a frame measure for infinitely many measures which are absolutely continuous with respect to Lebesgue measure. Theorem~\ref{3.3} is extended to every frame measure for $ \mu=\chi_Fd\lambda $, i.e., if $ F\subseteq \mathbb{R}^d $ and $ \nu $ is a frame measure for $ \mu=\chi_Fd\lambda $, then $ \nu $ is a frame measure for infinitely many measures which are absolutely continuous with respect to Lebesgue measure (Theorem \ref {3.4}). We show  applications of Theorem~\ref {3.4} in Examples \ref{7.3}, \ref{8.3}, \ref{9.3} and \ref{BX}. Similar to Theorem \ref{3.4}, in Proposition \ref{SF} we obtain that if $ \mu $ is a Borel measure on $ \mathbb{R}^d $ (not necessarily Lebesgue measure or absolutely continuous with respect to Lebesgue measure) and admits a frame measure $ \nu $, then infinitely many measures which are absolutely continuous with respect to $ \mu $ admit $ \nu $ as a frame measure. We apply Proposition \ref{SF} in Examples~\ref{EPX} and \ref{Exp} for invariant measures (Cantor type measures). Finally, in Corollary \ref{coo} we conclude that there are infinitely many absolutely continuous measures with respect to Lebesgue measure and infinitely many absolutely continuous measures with respect to a Cantor measure, which admit discrete and continuous frame measures.
\section{ Preliminaries}
\begin{definition}
Let $ H $ be a Hilbert space. A sequence $ \{f_i \}_{i \in I} $ of elements in $ H $  is called a \emph{frame} for $ H $, if there exist constants $ A, B > 0 $  such that for all $ f \in H $,
\begin{equation*}
 A\| f\|^2 \leq \sum_{i \in I }|\left<f, f_i\right>|^2 \leq B\| f\|^2.  
\end{equation*}
The constants $ A $ and $ B $ are called \emph{lower (frame) bound} and \emph{upper (frame) bound}, respectively. If $ A = B $, the frame is called \emph{tight} and whenever $ A = B = 1$, the frame is called \emph{Parseval}.

The sequence $ \{f_i \}_{i \in I} $ is called \emph{Bessel} if it has a finite upper frame
bound $ B $ and does not necessarily have a positive lower frame bound $ A $.
\end{definition} 
Frames are a natural generalization of orthonormal bases. The lower bound implies that a frame is complete in the Hilbert space, so by using (infinite) linear combination of the elements $ f_i $ in the frame every $ f $ can be expressed  \cite{1}.
\begin{definition}
Let $ t\in \mathbb{R}^d $. For every $ x\in \mathbb{R}^d$  the exponential function $ e_t $ is defined by $ e_t(x) = e^{2\pi it\cdot x} $. If $ \mu $ is a Borel measure on $ \mathbb{R}^d $, then for a function $ f\in L^1(\mu) $ the Fourier transform is given by
 \begin{align*}
\widehat{f d\mu}(t)=\int_{\mathbb{R}^d} f(x) e_{-t}(x) d\mu(x) \quad (t\in \mathbb{R}^d).
\end{align*}
\end{definition}
Note that whenever $ \mu $ is a finite measure, $ e_t\in L^2(\mu) $ and  $ \widehat{f d\mu}(t)=\left<f, e_t\right> $ for every $ f\in L^2(\mu) $. 
\begin{definition}
Let $ \mu $ be a finite Borel measure on $ \mathbb{R}^d $ and $\Lambda$ be a countable set in $ \mathbb{R}^d $. If the set $ E(\Lambda)=\{e_\lambda : \lambda \in \Lambda \} $ is a frame for $ L^2(\mu) $, then $ E(\Lambda) $ is called a \emph{Fourier frame}, $ \Lambda $ is called a \emph{frame spectrum} for $ \mu $ and $ \mu $ is called a \emph{frame spectral measure}. Likewise, if $ E(\Lambda) $ is an orthonormal basis (Bessel sequence) for $ L^2(\mu) $, then $ \Lambda $ is called a \emph{spectrum (Bessel spectrum)} for $ \mu $ and $ \mu $ is called a \emph{spectral measure (Bessel spectral measure)}. 
\end{definition}
We give the following definition from \cite{3}.
\begin{definition}[\cite{3}]
Let $ \mu $ be a Borel measure on $ \mathbb{R}^d $. A Borel measure $ \nu $ is called a \emph{frame measure} for $ \mu $ if there exist positive constants $ A, B $ such that for every $ f\in L^2(\mu) $,
\begin{equation}
A\| f\|_{L^2(\mu)}^2 \leq \int_{\mathbb{R}^d} |\widehat{fd\mu}(t)|^2 d\nu(t) \leq B\| f\|_{L^2(\mu)}^2.
\end{equation}
Here $ A $ and $ B $ are called \emph{(frame) bounds} for $ \nu $. The measure $ \nu $ is called a \emph{tight frame measure} if $ A = B $ and \emph{Plancherel measure} if $ A = B =1 $ (see also \cite{5}). If Equation (2.1) has upper bound $ B $ and does not necessarily have lower bound $ A $, then $ \nu $ is called a \emph{Bessel measure} for $ \mu $ and $ B $ is called a \emph{(Bessel) bound} for $ \nu $.

 Denote by $ \mathcal{B}_B(\mu) $ the set of all Bessel measures for $ \mu $ with fixed bound $ B $  and denote by $ \mathcal{F}_{A,B}(\mu) $ the set of all frame measures for $ \mu $ with fixed bounds $ A, B $.
\end{definition}
\begin{remark} \label{2.5}
A finite Borel measure $ \mu $ is a frame spectral measure if and only if there exists a countable set $\Lambda $ in $ \mathbb{R}^d $ such that $ \nu=\sum_{\lambda \in \Lambda}\delta_\lambda $ is a frame measure for $ \mu $.
\end{remark}
\begin{definition}
 A finite set of contraction maps $ \{\tau_i\}_{i=1}^n $ on a complete metric space is called an \emph{iterated function system (IFS)}. Hutchinson \cite{9} proved that there exists a unique compact subset $ X $ of $ \mathbb{R}^d $ and an invariant measure $ \mu $ (a unique Borel probability measure supported on $ X $) such that $ X=\bigcup_{i=1}^n \tau_i(X) $ and $ \mu=\sum_{i=1}^n \rho_i (\mu \circ\tau^{-1}) $, where $ 0 <\rho_i < 1 $,  $ \sum_{i=1}^n \rho_i =1 $. This measure $ \mu $ is either absolutely continuous or singular continuous with respect to Lebesgue measure. In an affine IFS each $ \tau_i $ is affine and represented by a matrix. Let $R$ be a $ d \times d $ expanding integer matrix (i.e., all eigenvalues have modules strictly greater than $ 1 $) and let $ \mathcal{A} $ be a finite subset of $ \mathbb{Z}^d $ of cardinality $\#\mathcal{A} =: N $. Then the following set is an affine iterated function system:
\begin{equation*}
\tau_a (x) = R^{-1}(x + a) \quad (x \in \mathbb{R}^d, a \in \mathcal{A}).
\end{equation*}
Taking $R$ as an expanding matrix guarantees that all maps $\tau_a$ are contractions (in an appropriate metric equivalent to the Euclidean one). Invariant measures on Cantor type sets (Cantor type measures), which are singular continuous with respect to Lebesgue measure, are examples of invariant measures of affine IFSs (see \cite{8, 9}). 
\end{definition} 
All measures we consider in this paper are Borel measures on $ \mathbb{R}^d $. We denote Lebesgue measure by $ \lambda $ and for any set $E \subset \mathbb{R}^d $, $|E|$ denotes the Lebesgue measure of $ E $.
\section{Investigation and Construction}
In this section we examine Bessel/frame measures and we prove some results concerning measures which admit Bessel/frame measures. 
\begin{proposition}\label{p!}
\emph{Let $ \mu $ be a finite measure. Then every finite measure $ \nu $ is a Bessel measure for $ \mu $}.
\end{proposition} 
\begin{proof}
Let $ f\in L^2 (\mu) $ and $ t\in\mathbb{R}^d $. Using Holder's inequality, we have
\begin{equation*}
|\left<f, e_t\right>| \leq \int_{\mathbb{R}^d} \left| f(x) e_{-t}(x) \right| d\mu(x) \leq \left(\mu(\mathbb{R}^d)\right)^{\frac{1}{2}} \| f\|_{L^2(\mu)}.
\end{equation*}
Then
\begin{equation*}
\int_{\mathbb{R}^d}|\left<f, e_t\right>|^2 d\nu(t) \leq \mu(\mathbb{R}^d) \nu(\mathbb{R}^d) \| f\|^2_{L^2(\mu)}.
\end{equation*}
Hence $ \nu \in \mathcal{B}_{\mu(\mathbb{R}^d) \nu(\mathbb{R}^d)}(\mu) $.
\end{proof}
\begin{remark}\label{RE}
The above proposition shows that the Bessel bound may change for different Bessel measures $ \nu $, but for probability measures $ \nu $ we have $ \nu \in \mathcal{B}_{\mu(\mathbb{R}^d)}(\mu) $. Note that there are infinitely many probability measures $ \nu $ (such as every measure $ \frac{1}{\lambda(E)}\chi_E d\lambda $ where $ E \subset \mathbb{R}^d$ with the finite Lebesgue measure $ \lambda(E) $, every finite discrete measure $ \frac{1}{n}\sum_{a=1}^n \delta_a $ where $ \delta_a $ denotes the Dirac measure at the point $ a $, every invariant measure obtained from an iterated function system, and others), so $ \mathcal{B}_{\mu(\mathbb{R}^d)}(\mu) $ is an infinite set.
\end{remark}
\begin{proposition}\label{!p}
 \emph{Let $ \nu $ be a finite measure. Then $ \nu $ is a Bessel measure for every finite measure $ \mu $. In particular, $ \nu \in \mathcal{B}_{\nu(\mathbb{R}^d)}(\mu) $ for all probability measures $ \mu $}.
\end{proposition}
\begin{proof}
The proof is similar to the proof of Proposition \ref{p!}. 
\end{proof}
\begin{corollary}\label{MK}
\emph{A finite measure $ \nu $ is a Bessel measure for a finite measure $ \mu $, if and only if $ \mu $ is a Bessel measure for $ \nu $. Consequently, every finite measure $ \mu $ is a Bessel measure to itself}.
\end{corollary}
\begin{proof}
The conclusion follows from Propositions \ref{p!} and \ref{!p}.
\end{proof}

(see also the extended form of the above assertions in our recent work \cite{19})

\begin{proposition}[\cite{3}]
\emph{Let $ \mu $ be a finite measure and let $ B $ be a positive constant. Then there exists a Bessel measure $ \nu $ for $ \mu $ which is not necessarily finite}.
\end{proposition}
\begin{proof}
For a countable set $ \Lambda\subset\mathbb{R}^d $ let $ \nu=\sum_{\lambda\in\Lambda}c_\lambda\delta_\lambda $ such that $ \sum_{\lambda\in\Lambda}c_\lambda \leq \dfrac{B}{\mu(\mathbb{R}^d)} $. Then by applying Holder's inequality one can obtain
\begin{equation*}
\int_{\mathbb{R}^d} |\left<f, e_t \right>|^2 d\nu(t) \leq \sum_{\lambda\in\Lambda}c_\lambda \|f\|^2_{L^2(\mu)} \mu(\mathbb{R}^d)\leq B\|f\|^2_{L^2(\mu)} \quad \text{for all }  f\in L^2(\mu).
\end{equation*}
\end{proof}
\begin{proposition}[\cite{3}]\label{Aj}
\emph{If $ \nu $ is a Bessel measure for a finite measure $ \mu $, then $ \nu $ is a $ \sigma $-finite measure}.
\end{proposition}
\begin{proposition}\label{11.3}
\emph{If a $ \sigma $-finite measure $ \nu $ is a Bessel measure for $ \mu_1 $, $ \mu_2 $, then $ \nu $ is a Bessel measure for $ \mu_1 + \mu_2 $}.
\end{proposition}
\begin{proof}
Let $ B_1 $, $ B_2 $ be the Bessel bound for $ \nu $ (associated to $\mu_1 $, $ \mu_2 $ respectively). If we apply Holder's inequality, then for all $ f \in L^2( \mu_1 + \mu_2) $,
\begin{align*}
\int_{\mathbb{R}^d} |\widehat{fd(\mu_1 + \mu_2)}|^2 d\nu &\leq B_1 \|f\|^2_{\mu_1} + B_2 \|f\|^2_{\mu_2} + 2\sqrt{B_1 B_2} \|f\|_{\mu_1} \|f\|_{\mu_2}\\
&\leq (\sqrt{B_1} + \sqrt{B_2})^2 \|f\|^2_{\mu_1 + \mu_2}.
\end{align*}
Thus, the assertion follows.
\end{proof}
Note that when $\mu_1 $, $\mu_2$ and $ \nu $ are finite measures, by Proposition \ref{p!} there exists a Bessel bound $ (\mu_1 + \mu_2)(\mathbb{R}^d)\nu(\mathbb{R}^d) $ for $ \nu $.

Now the question is whether there is a connection between the existence of a frame measure for $ \mu $, $ \mu' $ and the sum $ \mu + \mu' $. We give the following lemma from \cite{6} (see also Proposition \ref{frb}).
\begin{lemma}[\cite{6}]
\emph{Let $ \mu $, $ \mu' $ be Borel measures. Suppose that $ \mu'(K_{\mu}) = 0 $ ($ K_{\mu} $ is the smallest closed set such that $ \mu(K) = \mu(\mathbb{R}^d $)). If $ \nu $ is a frame measure for $ \mu + \mu' $, then $ \nu $ is a frame measure for $ \mu $ and $ \mu' $ with the same frame bounds}.
\end{lemma}
\begin{proposition}\label{3.1}
\emph{Let $ \mu $ be a Borel measure supported on $ F\subseteq\mathbb{R}^d $ and $\nu \in \mathcal{F}_{A,B}(\mu)$. If $ E\subseteq F $ and $ 0 < m \leq \phi(x)\leq M <\infty $ $ \mu $-a.e. on $ E $, then $ \nu $ is a frame measure for $ \mu'=\chi_{E} \phi d\mu$. More precisely, $\nu \in \mathcal{F}_{mA,MB}(\mu')$}.
\end{proposition} 
\begin{proof}
 Since $ \nu $ is a frame measure for $ \mu $, for every $f \in L^2 (\mu)$,  
\begin{equation*}
A\| f\|^2_{L^2(\mu)} \leq \int_{\mathbb{R}^d} \left|\int_{\mathbb{R}^d} f(x) e_{-t}(x) d\mu(x) \right|^2 d\nu(t)\leq B\| f\|^2_{L^2(\mu)} .
\end{equation*}
In addition, for every $f \in L^2 (\mu)$ we have $\chi_{E}\phi f \in L^2 (\mu)$, since
\begin{equation*}
\int_{\mathbb{R}^d} |\chi_{E}(x)\phi(x)f(x)|^2 d\mu(x) = \int_{E} |\phi(x)|^2|f(x)|^2 d\mu(x)\leq M^2\int_{\mathbb{R}^d}|f(x)|^2 d\mu(x) < \infty.
 \end{equation*}
We have 
\begin{align*}
\int_{\mathbb{R}^d} \left|\int_{\mathbb{R}^d} f(x) e_{-t}(x) d\mu'(x) \right|^2 d\nu(t) & = \int_{\mathbb{R}^d} \left|\int_{\mathbb{R}^d} f(x) e_{-t}(x) \chi_{E}(x)\phi(x) d\mu(x) \right|^2 d\nu(t)\\
& \leq B \int_{\mathbb{R}^d} |\chi_{E}(x)\phi(x) f(x) |^2 d\mu(x)\\
& \leq BM \int_{\mathbb{R}^d} |f(x)|^2 \chi_{E}(x)\phi(x) d\mu(x) \\
&= BM \| f\|^2_{L^2(\mu')}.
\end{align*}
Analogously, we obtain the lower bound and consequently
\begin{equation*}
Am\| f\|^2_{L^2(\mu')} \leq \int_{\mathbb{R}^d} \left|\int_{\mathbb{R}^d} f(x) e_{-t}(x) d\mu'(x) \right|^2 d\nu(t) \leq BM\| f\|^2_{L^2(\mu')}.
\end{equation*}
\end{proof}
\begin{corollary}\label{Cco}
\emph{Let $ \mu $ be a Borel measure and let $\nu $ be a Bessel/frame measure for $\mu$. Then for any $ E\subset \mathrm{supp}\mu $, the measure $ \mu'=\chi_{E}d\mu $ admits $ \nu $ as a Bessel/frame measure with the same bound(s)}. 
\end{corollary}
\begin{proposition}\label{SA}
\emph{Let $ E\subseteq [0, 1]^d $ and $ 0 < m \leq \phi(x)\leq M <\infty $ $ \lambda $-a.e. on $ E $. Then the measure $ \nu=\sum_{t\in \mathbb{Z}^d} \delta_t $ is a Plancherel measure for $ \mu=\chi_{E}d\lambda $ and a frame measure for $ \mu'=\chi_{E}\phi d\lambda $. Precisely, $ \nu \in \mathcal{F}_{1,1}(\mu) $ and $ \nu \in \mathcal{F}_{m,M}(\mu') $}.
\end{proposition}
\begin{proof}
Since $ \{e_t\}_{t \in \mathbb{Z}^d} $ is an orthonormal basis for $ L^2([0, 1]^d) $, 
\begin{equation*}
\sum_{t\in \mathbb{Z}^d}|\left<f, e_t \right>|^2 =\int_{[0, 1]^d} |f(x)|^2 d\lambda(x) \quad \text{for all }  f\in L^2([0, 1]^d).
\end{equation*}
Considering $ \mu = \chi_{\{[0, 1]^d\}}d\lambda $ on $ \mathbb{R}^d $, we have for all $ f\in L^2(\mu) $,
\begin{equation*}
\int_{\mathbb{R}^d} |\left<f, e_y \right>_{L^2(\mu)}|^2 d\nu(y)= \sum_{t\in \mathbb{Z}^d}|\left<f, e_t \right>_{L^2(\mu)}|^2 = \int_{\mathbb{R}^d} |f(x)|^2 d\mu(x).
\end{equation*}
Then the assertion follows from Proposition \ref{3.1} and Corollary \ref{Cco}.
\end{proof}
\begin{proposition}\label{3.2}
\emph{Let $ F\subseteq \mathbb{R}^d $ and $ 0 < m \leq \phi(x)\leq M<\infty $ $ \lambda $-a.e. on $ F $. Then  $ \lambda $ is a Plancherel measure for $ \mu=\chi_{F}d\lambda $ and a frame measure for $ \mu'=\chi_{F}\phi d\lambda $. Precisely, $ \lambda \in \mathcal{F}_{1,1}(\mu) $ and $ \lambda \in \mathcal{F}_{m,M}(\mu') $}.
\end{proposition}
\begin{proof}
According to Plancherel's theorem the following equation is satisfied:
\begin{equation*}
\int_{\mathbb{R}^d} \left|\int_{\mathbb{R}^d}f(x)e_{-t}(x)d\lambda(x)\right|^2 d\lambda(t) =\int_{\mathbb{R}^d} |f(x)|^2 d\lambda(x) \quad \text{for all }  f\in L^2(\lambda).
\end{equation*}
Then the assertion follows from Proposition \ref{3.1} and Corollary \ref{Cco}.
\end{proof}
In the following theorem we construct infinitely many measures which admit Lebesgue measure as a frame measure with arbitrary fixed frame bounds $ m,M $. 
\begin{theorem} \label{3.3}
\emph{Lebesgue measure is a frame measure for infinitely many measures which are absolutely continuous with respect to Lebesgue measure}.
\end{theorem}
\begin{proof}
We first recall that for measurable functions $ f, g $ on $ \mathbb{R}^d $, if $ \mu= fd\lambda $ and $ \nu= gd\lambda $, then we have $ \mu\ast \nu = (f \ast g) d\lambda $. Now let $ 0 < m \leq \phi(x)\leq M<\infty $ $ \lambda $-a.e. on $ \mathbb{R}^d $. For every $ n\in\mathbb{N}$ with $ 1\leq n\leq N $, let $ E_n \subset \mathbb{R}^d $, $ \lambda(E_n) < \infty $ and $ \mu_n= \frac{1}{\lambda(E_n)}\chi_{E_n} d\lambda $. Take $ \mu_0= \phi d\lambda $. Then $ \mu_0 \ast \mu_1= (\phi \ast \frac{1}{\lambda(E_1)}\chi_{E_1}) d\lambda $ and 
\begin{equation*}
 \phi \ast \frac{1}{\lambda(E_1)}\chi_{E_1}(x) = \int_{\mathbb{R}^d} \phi(x-y)\frac{1}{\lambda(E_1)}\chi_{E_1}(y)d\lambda(y) \leq M.
\end{equation*}
Similarly, we obtain $ m $ as a lower bound, i.e.,  
 \begin{equation}
m \leq \phi \ast \frac{1}{\lambda(E_1)}\chi_{E_1} \leq M.
\end{equation}
By Plancherel's theorem 
\begin{equation*}
\int_{\mathbb{R}^d} \left|\int_{\mathbb{R}^d}f(x)e_{-t}(x)d\lambda(x)\right|^2 d\lambda(t) =\int_{\mathbb{R}^d} |f(x)|^2 d\lambda(x) \quad \text{for all }  f\in L^2(\lambda),
\end{equation*}
so $ \lambda $ is a Plancherel measure to itself, and by (3.1), for every $ f\in L^2(\lambda) $ we have $ (\phi \ast \frac{1}{\lambda(E_1)}\chi_{E_1}) f \in L^2(\lambda) $. 
Hence for all $ f\in L^2(\mu_0\ast\mu_1) $, 
\begin{align*}
\int_{\mathbb{R}^d} \left|\widehat{fd(\mu_0 \ast \mu_1)}(t)\right|^2 d\lambda(t) &= \int_{\mathbb{R}^d} \left|\int_{\mathbb{R}^d}(\phi \ast \frac{1}{\lambda(E_1)}\chi_{E_1})(x) f(x) e_{-t}(x)d\lambda(x)\right|^2 d\lambda(t)\\
&=\int_{\mathbb{R}^d} |(\phi \ast \frac{1}{\lambda(E_1)}\chi_{E_1})(x) f(x)|^2 d\lambda(x)\\
&\leq M \int_{\mathbb{R}^d} |f(x)|^2 (\phi \ast \frac{1}{\lambda(E_1)}\chi_{E_1})d\lambda(x)\\
&= M\| f\|^2_{L^2(\mu_0\ast\mu_1)}.
\end{align*}
Analogously, we obtain the lower bound and consequently, for all $ f\in L^2(\mu_0\ast\mu_1) $,
\begin{equation*}
m\| f\|^2_{L^2(\mu_0\ast\mu_1)} \leq \int_{\mathbb{R}^d} \left|\widehat{fd(\mu_0 \ast \mu_1)}(t)\right|^2 d\lambda(t) \leq M\| f\|^2_{L^2(\mu_0\ast\mu_1)}.
\end{equation*}
Likewise, convolution of measures $ \mu_0\ast\mu_1 $ and $ \mu_2= \frac{1}{\lambda(E_2)}\chi_{E_2}d\lambda $ yields $ \lambda \in \mathcal{F}_{m,M}(\mu_0\ast\mu_1\ast\mu_2)$, and repeating this process gives the assertion. Precisely, for any $ n\in\mathbb{N}$ with $ 1\leq n\leq N $, one can obtain $ \lambda \in \mathcal{F}_{m,M}(\mu_0\ast\mu_1\ast\mu_2\cdots\ast\mu_n)$.
\end{proof}
 We proved Theorem \ref{3.3} considering the fact that $ \lambda $ is a Plancherel measure to itself (Plancherel theorem). In the next theorem we show that if $ F\subseteq \mathbb{R}^d $, $ \mu=\chi_Fd\lambda $, then Theorem \ref{3.3} can be extended to every frame measure for $ \mu $.
\begin{theorem} \label{3.4}
\emph{Let $ F\subseteq \mathbb{R}^d $ and let $ \nu $ be a frame measure for $ \mu=\chi_Fd\lambda $ with bounds $ A, B $. Then $ \nu $ is a frame measure for infinitely many measures which are absolutely continuous with respect to Lebesgue measure}.
\end{theorem}
\begin{proof}
Let $ 0 < m \leq \phi(x)\leq M <\infty $ $ \mu $-a.e. on $ \mathbb{R}^d $. For every $ n\in\mathbb{N}$ with $ 1\leq n\leq N $, let $ E_n \subset \mathbb{R}^d $, $ \lambda(E_n) < \infty $ and $ \mu_n= \frac{1}{\lambda(E_n)}\chi_{E_n} d\lambda $. If $ \mu_0= \phi d\lambda $, then $ \mu_0 \ast \mu_1= (\phi \ast \frac{1}{\lambda(E_1)}\chi_{E_1}) d\lambda $. We have 
 \begin{equation*}
m \leq \phi \ast \frac{1}{\lambda(E_1)}\chi_{E_1} \leq M,
\end{equation*}
and for every $ f\in L^2(\chi_Fd(\mu_0 \ast \mu_1)) $ we have
\begin{equation*}
\int_{\mathbb{R}^d} \left|\widehat{f\chi_Fd(\mu_0 \ast \mu_1)}(t)\right|^2d\nu(t) = \int_{\mathbb{R}^d} \left|\int_{\mathbb{R}^d} f(x) e_{-t}(x) (\phi \ast \frac{1}{\lambda(E_1)}\chi_{E_1})(x) d\mu(x)\right|^2 d\nu(t). 
\end{equation*}
Since $ \nu $ is a frame measure for $ \mu $ and $ (\phi \ast \frac{1}{\lambda(E_1)}\chi_{E_1}) f \in L^2(\mu) $, 
\begin{align*}
\int_{\mathbb{R}^d} \left|\int_{\mathbb{R}^d} f(x) e_{-t}(x) (\phi \ast \frac{1}{\lambda(E_1)}\chi_{E_1})(x) d\mu(x)\right|^2 d\nu(t) &\leq B\int_{\mathbb{R}^d}|(\phi \ast \frac{1}{\lambda(E_1)}\chi_{E_1})(x) f(x)|^2 d\mu(x) \\
& \leq BM  \int_{\mathbb{R}^d} |f(x)|^2 (\phi \ast \frac{1}{\lambda(E_1)}\chi_{E_1})(x) d\mu(x) \\
& = BM  \int_{\mathbb{R}^d} |f(x)|^2 (\phi \ast \frac{1}{\lambda(E_1)}\chi_{E_1})(x) \chi_F(x) d\lambda(x) \\
&= BM \| f\|^2_{L^2(\chi_Fd(\mu_0 \ast \mu_1))}
\end{align*}
Similarly, we obtain $ Am $ as a lower bound. Hence for all $ f\in L^2(\chi_Fd(\mu_0 \ast \mu_1)) $,
\begin{equation*}
mA\| f\|^2_{L^2(\chi_Fd(\mu_0\ast\mu_1))} \leq \int_{\mathbb{R}^d} \left|\widehat{f\chi_Fd(\mu_0 \ast \mu_1)}(t)\right|^2 d\nu(t) \leq MB\| f\|^2_{L^2(\chi_Fd(\mu_0\ast\mu_1))}.
\end{equation*}
 Convolution of measures $ \mu_0\ast\mu_1 $ and $ \mu_2= \frac{1}{\lambda(E_2)}\chi_{E_2}d\lambda $ yields $ \nu \in \mathcal{F}_{mA,MB}(\chi_Fd(\mu_0\ast\mu_1\ast\mu_2))$. Likewise, for any $ n\in\mathbb{N}$ with $ 1\leq n\leq N $, one can obtain $ \nu \in \mathcal{F}_{mA, MB}(\chi_Fd(\mu_0\ast\mu_1\ast\mu_2\cdots\ast\mu_n))$, and then the theorem follows.
\end{proof}
\begin{remark}\label{REE}
In Theorems \ref{3.3}, \ref{3.4}, if any of the measures $ \mu_n = \frac{1}{\lambda(E_n)}\chi_{E_n} $ changes to $ \mu_n = \chi_{E_n} $, then the bounds is multiplied by $ \lambda(E_n) $.
\end{remark}
\begin{example}\label{7.3}
Let $ F\subseteq\mathbb{R}^d $ and for every $ n\in\mathbb{N}$ with $ 1\leq n\leq N $, let $ E_n \subset \mathbb{R}^d $, $ \lambda(E_n) < \infty $. By Proposition~\ref{3.2}, $ \lambda $ is a Plancherel measure for $ \mu=\chi_{F}d\lambda $ then by Theorem \ref{3.4}, $ \lambda $ is a Plancherel measure for $\chi_F d(\lambda\ast\frac{1}{\lambda(E_1)}\chi_{E_1}d\lambda\ast\cdots\ast\frac{1}{\lambda(E_n)}\chi_{E_n}d\lambda) $.
\end{example}
\begin{example}\label{8.3}
For every $ n\in\mathbb{N}$ with $ 1\leq n\leq N $, let $ E_n \subset \mathbb{R}^d $, $ \lambda(E_n) < \infty $. By Proposition \ref{SA}, the measure $ \nu=\sum_{t \in \mathbb{Z}^d} \delta_t $ is a Plancherel measure for $ \mu= \chi_{\{[0 , 1]^d\}}d\lambda $ and also by Theorem \ref{3.4}, $ \nu $ is a Plancherel measure for $\chi_{\{[0, 1]^d\}}d( \lambda\ast\frac{1}{\lambda(E_1)}\chi_{E_1}d\lambda\ast\cdots\ast\frac{1}{\lambda(E_n)}\chi_{E_n}d\lambda) $. 
\end{example}
To show another application of Theorem \ref {3.4} we need the following theorem.  
\begin{theorem}[\cite{16}] \label{3.9}
\emph{There exist positive constants $ c, C $ such that for every set $ E \subset \mathbb{R}^d $ of finite Lebesgue measure, there is a discrete set $ \Lambda \subset \mathbb{R}^d $ such that $ \nu=\sum_{t \in \Lambda} \delta_t $ is a frame measure for $ L^2(\chi_{E}d\lambda) $ with frame bounds $ c|E| $ and $ C|E| $}.
\end{theorem}
\begin{example}\label{9.3}
Let $ E\subset\mathbb{R}^d $ and for every $ n\in\mathbb{N}$ with $ 1\leq n\leq N $, let $ E_n \subset \mathbb{R}^d $, $ \lambda(E_n) < \infty $. By Theorems \ref{3.9} and \ref {3.4}, $ \nu=\sum_{t \in \Lambda} \delta_t $ is a frame measure for $\chi_E d(\lambda\ast\frac{1}{\lambda(E_1)}\chi_{E_1}d\lambda\ast\cdots\ast\frac{1}{\lambda(E_n)}\chi_{E_n}d\lambda) $ with frame bounds $ c|E| $ and $ C|E| $. 
\end{example}
\begin{proposition}[\cite{3}]\label{5.3}
\emph{Let $A$ and $B$ be fixed positive constants and $ \mu $ be a finite measure. Then the set of all Bessel measures for $ \mu $ with bound $B$ (or $ \mathcal{B}_B(\mu) $) and the set of all frame measures for $ \mu $ with bounds $A$, $B$ (or $ \mathcal{F}_{A, B}(\mu) $), are convex and closed under convolution with Borel probability measures}.
\end{proposition}
\begin{remark}
Since the set of Bessel/frame measures (for a fixed measure $ \mu $) is closed under convolution with Borel probability measures, if a measure $ \nu $ is a Bessel/frame measure for $ \mu $, then considering Proposition \ref{Aj}, one can obtain infinitely many $ \sigma $-finite measures. In fact, there are infinitely many probability measures $ \rho $ (as we mentioned in Remark \ref{RE}) and one can convolute $ \nu $ with every one of them many times.
\end{remark}
\begin{example}\label{!!}
Based on Examples \ref{7.3} and \ref{8.3}, $ \lambda $ and the discrete measure $ \nu=\sum_{t \in \mathbb{Z}^d} \delta_t $ are in $ \mathcal{F}_{1,1}\left(\chi_{\{[0, 1]^d\}}d(\lambda\ast \frac{1}{\lambda(E_1)}\chi_{E_1}d\lambda\ast\cdots\ast\frac{1}{\lambda(E_n)}\chi_{E_n}d\lambda) \right)$. Since by proposition \ref{5.3}, the set of all frame measures are convex, we have $ \frac{1}{2}(\lambda + \nu) \in  \mathcal{F}_{1,1}\left(\chi_{\{[0, 1]^d\}}d( \lambda\ast\frac{1}{\lambda(E_1)}\chi_{E_1}d\lambda\ast\cdots\ast\frac{1}{\lambda(E_n)}\chi_{E_n}d\lambda) \right)$.
\end{example}
\begin{example}\label{BX}
Let $ \mathcal{P}(\mathbb{R}^d) $ be the set of all probability measures on $ \mathbb{R}^d $ and for every $ n\in\mathbb{N}$ with $ 1\leq n\leq N $, let $ \rho_n, \rho'_n \in \mathcal{P}(\mathbb{R}^d) $. We have $ \lambda $ and $ \nu=\sum_{t \in \mathbb{Z}^d} \delta_t $ are in $ \mathcal{F}_{1,1}(\chi_{\{[0, 1]^d\}}d\lambda) $ (see Propositions~\ref{3.2} and \ref{SA}). By Proposition \ref{5.3} the set $ \mathcal{F}_{1, 1}(\chi_{\{[0, 1]^d\}}d\lambda) $ is closed under convolution with Borel probability measures, so for all $ n\in \{1,\ldots, N \} $, we have $ \lambda\ast\rho_1\ast\cdots\ast\rho_n $ and $ \nu \ast\rho'_1\ast\cdots\ast\rho'_n $ and the convex combinations of all these measures are in $ \mathcal{F}_{1,1}(\chi_{\{[0, 1]^d\}}d\lambda) $. In addition, by Theorem \ref{3.4}, for all $ n\in \{1,\ldots, N \} $ we have $ \lambda\ast\rho_1\ast\cdots\ast\rho_n $ and $ \nu \ast\rho'_1\ast\cdots\ast\rho'_n $ and also the convex combinations of all these measures are in $ \mathcal{F}_{1,1}\left(\chi_{\{[0, 1]^d\}}d( \lambda\ast\frac{1}{\lambda(E_1)}\chi_{E_1}d\lambda\ast\cdots\ast\frac{1}{\lambda(E_n)}\chi_{E_n}d\lambda) \right) $. 
\end{example}
\begin{remark}
Note that by Proposition \ref{3.1} we can construct new measures such that admit all frame measures in Example \ref{BX} as frame measures, considering the fact that for any $ E\subset [0, 1]^d $ we have $ \lambda $ and $ \nu=\sum_{t \in \mathbb{Z}^d} \delta_t $ are in $ \mathcal{F}_{1,1}(\chi_Ed\lambda) $.
\end{remark}
\begin{definition}[\cite{3}]
A sequence of Borel probability measures $ \{\rho_n\}_{n\in \mathbb{N}} $ is called an \emph{approximate identity} if 
\begin{equation*}
 sup \{ \parallel t\parallel : t\in \mathrm{supp} (\rho_n)\} \rightarrow 0  \quad \text{as}\quad  n\rightarrow \infty.
\end{equation*} 
\end{definition}
\begin{example}
Some approximate identities on $ \mathbb{R}^d $ are:

(i) $ \rho_n = n^d \chi_{\{[0, \frac{1}{n}]^d\}} $.

(ii) $ \rho_n = (\dfrac{n}{2})^d \chi_{\{[-\frac{1}{n}, \frac{1}{n}]^d\}} $.

(iii) $ \rho_n = (n(n+1))^d \chi_{\{[\frac{1}{n+1}, \frac{1}{n}]^d\}} $.

(iv) $ \rho_n = 2^{(n-1)d} \chi_{\{[0, \frac{1}{2^{n-1}}]^d\}} $ (in general, for $ m\in \mathbb{N} $, $ m\geq 2 $, $ \lambda_n = m^{(n-1)d} \chi_{\{[0, \frac{1}{m^{n-1}}]^d\}} $).
\end{example}
By Proposition \ref{5.3}, if $ \nu $ is a Bessel/frame measure for $ \mu $, then for any probability measure $ \rho $, the measure $ \nu\ast\rho $ is also a Bessel/frame measure for $ \mu $. To see under what conditions the converse is true we give the following theorem from \cite{3}.  
\begin{theorem}[\cite{3}]
\emph{Let $ \{\rho_n\} $ be an approximate identity. Suppose $ \nu $ is a $ \sigma $-finite measure and suppose $ \nu \ast \rho_n $ are Bessel/frame measures for $ \mu $ with uniform bounds, independent of $ n $. Then $ \nu $ is a Bessel/frame measure for $ \mu $}.
\end{theorem}
\begin{lemma} \label{Le.2}
\emph{Let $ \nu \in \mathcal{F}_{A, B}(\mu) $. Let $ 0 < m\leq \phi(x)\leq M <\infty$, $ \mu $-a.e. on $ \mathbb{R}^d $. Then $ \nu \in \mathcal{F}_{mA, MB}(\phi d\mu) $}.
\end{lemma}
\begin{proof}
For every $ f\in L^2(\phi d\mu) $,
\begin{equation*}
 |\widehat{f(\phi d\mu)}| = |\widehat{\phi fd\mu}|, 
\end{equation*}
and
\begin{equation*}
m\int_{\mathbb{R}^d} |f|^2 \phi d\mu \leq \int_{\mathbb{R}^d}|\phi f|^2 d\mu \leq M \int_{\mathbb{R}^d} |f|^2 \phi d\mu.
\end{equation*}
So, we obtain
\begin{equation*}
 mA\int_{\mathbb{R}^d} |f|^2 \phi d\mu \leq \int_{\mathbb{R}^d}|\widehat{f(\phi d\mu)}|^2 d\nu(t) \leq MB \int_{\mathbb{R}^d} |f|^2 \phi d\mu.  
\end{equation*}
\end{proof}
\begin{corollary} \label{a.1}
\emph{If $ \nu \in \mathcal{F}_{A,B}(\mu) $, then for any constant $\alpha > 0$, $ \nu $ is a frame measure for $ \alpha\mu $. More precisely, $ \nu \in \mathcal{F}_{\alpha A,\alpha B}(\alpha\mu) $}.
\end{corollary}
\begin{proposition}\label{frb}
\emph{Let $ \mu $ be a Borel measure supported on $ F\subseteq\mathbb{R}^d $. For every $ n\in\mathbb{N}$ with $ 1\leq n\leq N $, let $ L_n\subset \mathbb{R}^d $, $ \mu(L_n) = 0 $, $ E_n= F\setminus L_n $. Suppose that $ 0<m_n\leq \phi_n(x) \leq M_n<\infty $ $ \mu-a.e $ on $ E_n $ and $ \mu_n= \chi_{E_n}\phi_n(x)d\mu(x) $. If a $ \sigma $-finite measure $ \nu $ is a frame  measure for $ \mu$, then $ \nu $ is a frame  measure for $ \mu +\mu_1 +\cdots+ \mu_n $}.
\end{proposition}
\begin{proof}
Since $ \nu $ is a frame measure for $ \mu $, by Proposition \ref{3.1}, $ \nu $ is also a frame measure for $ \mu_n $ for all $ n\in \{1, \ldots, N \} $. Let $ A, B $ be the bounds for $ \nu $. We have
\begin{equation*}
\mu' := \mu +\mu_1 +\cdots+ \mu_n = (1+\phi_1+ \cdots +\phi_n) d\mu  \;\;\;\;\; \mu-a.e.
\end{equation*} 
Then by Lemma \ref{Le.2}, for all $ f \in L^2( \mu' ) $, 
\begin{equation*}
(1+m_1+ \cdots + m_n)A \|f\|^2_{\mu'}  \leq \int_{\mathbb{R}^d} |\widehat{fd\mu'}|^2 d\nu \leq (1+M_1+ \cdots + M_n)B \|f\|^2_{\mu'}.
\end{equation*}
Hence, we have the desired result.
\end{proof}
\begin{lemma}\label{123} 
\emph{Let $ \nu \in \mathcal{F}_{A, B}(\mu) $. Let $ 0 < m\leq \phi(x)\leq M<\infty $, $ \nu $-a.e. on $ \mathbb{R}^d $. Then $ \phi d\nu \in \mathcal{F}_{mA, MB}(\mu) $ and consequently, for any constant $\alpha > 0$, $ \alpha\nu \in \mathcal{F}_{\alpha A, \alpha B}(\mu) $}.
\end{lemma}
\begin{proof}
 Since $ \nu $ is a frame measure for $ \mu $, the lemma follows directly from the definition.
\end{proof}
\begin{remark}\label{**}
Note that if $ \nu \in \mathcal{F}_{A, B}(\mu) $ and $ \nu' \in \mathcal{F}_{A', B'}(\mu) $, then for any two positive constants $ \alpha, \beta $, we have $ \alpha\nu + \beta\nu' \in \mathcal{F}_{\alpha A+\beta A', \alpha B+\beta B'}(\mu) $. Besides, if $ \nu, \nu'\in \mathcal{F}_{A, B}(\mu) $ we know from Proposition \ref{5.3}, $ \alpha\nu +(1-\alpha)\nu' \in \mathcal{F}_{A, B}(\mu) $.
\end{remark}
\begin{remark}
Let $ \mu $ be a Borel measure, and let $ \rho $ be a probability measure. Suppose $ 0 < m\leq \phi(x)\leq M<\infty $ on $ \mathbb{R}^d $. We have $ m\leq \phi \ast\rho\leq M $, Since $ (\phi \ast \rho)(x) = \int_{\mathbb{R}^d } \phi(x-y)d\rho(y) $. Hence by Lemma \ref{123}, if $ \nu \in \mathcal{F}_{A, B}(\mu) $, then $ (\phi \ast \rho) d\nu \in \mathcal{F}_{mA, MB}(\mu) $, and by Lemma \ref{Le.2}, if $ \nu \in \mathcal{F}_{A, B}(\mu) $, then $ \nu \in \mathcal{F}_{mA, MB}((\phi \ast \rho) d\mu) $.
\end{remark}
In the following we give a proposition similar to Theorem \ref{3.4} showing that if $ \mu $ is a Borel measure (not necessarily Lebesgue measure or absolutely continuous with respect to Lebesgue measure) and admits a frame measure $ \nu $, then infinitely many measures which are absolutely continuous with respect to $ \mu $ admit $ \nu $ as a frame measure.
\begin{proposition}\label{SF}
\emph{Suppose $ \mu $ is a Borel measure and $ \nu \in \mathcal{F}_{A, B}(\mu) $. Let $ 0<m\leq \phi(x) \leq M <\infty$ on $ \mathbb{R}^d $, and for every $ n\in\mathbb{N}$ with $ 1\leq n\leq N $, let $ \rho_n $ be a probability measure. Then $ \nu $ is a frame measure for all measures $ \phi\ast\rho_1\ast\cdots\ast\rho_n d\mu$}.
\end{proposition}
\begin{proof}
For $ n\in \{1,\ldots, N \} $, we have $ m \leq \phi \ast\rho_1\ast\cdots\ast\rho_n
\leq M $. Then by Lemma \ref{Le.2},
\begin{equation*}
m\| f\|^2_{L^2(\phi \ast\rho_1\ast\cdots\ast\rho_nd\mu)} \leq \int_{\mathbb{R}^d} \left|\widehat{fd(\phi \ast\rho_1\ast\cdots\ast\rho_n d\mu)}(t)\right|^2 d\nu(t) \leq M\| f\|^2_{L^2(\phi \ast\rho_1\ast\cdots\ast\rho_nd\mu)},
\end{equation*}
 for all $ f \in L^2(\phi \ast\rho_1\ast\cdots\ast\rho_n d\mu) $. Therefore, $ \nu \in \mathcal{F}_{mA, MB}(\phi \ast\rho_1\ast\cdots\ast\rho_n d\mu) $.
\end{proof}
Any fractal measure $ \mu $ obtained from an affine iterated function system has a discrete Bessel measure $ \nu=\sum_{\lambda \in \Lambda_\mu} \delta_\lambda $ (see \cite{2}). Moreover, when $ \mu $ is a Cantor type measure with even contraction ratio, $ \nu =\sum_{\lambda \in \Lambda_\mu} \delta_\lambda $ is a Plancherel measure for $\mu$, i.e., $ \nu \in \mathcal{F}_{1, 1}(\mu) $ (see \cite{10}).
\begin{example}\label{EPX}
Let $ \mu $ be a Cantor type measure with even contraction ratio and let $\nu =\sum_{\lambda \in \Lambda_\mu} \delta_\lambda $ be its associated Plancherel measure . For every $ n\in\mathbb{N}$ with $ 1\leq n\leq N $, let $ \rho_n $ be a probability measure. Suppose $ 0 < m\leq \phi(x)\leq M $ on $ \mathbb{R}^d $. Then by Proposition \ref{SF} we have $ \nu \in \mathcal{F}_{m,M}(\phi\ast\rho_1\ast\cdots\ast\rho_nd\mu) $.
\end{example}
\begin{example}\label{Exp}
Let $ \mu_4 $, $ \mu'_4 $ be the invariant measures (Cantor measures) for the affine IFSs with $ R = 4 $, $ \mathcal{A} = \{0, 2\} $, and $ R = 4 $, $ \mathcal{A}' = \{0, 1\} $ respectively. Then by Corollary $ 4.7 $ from \cite{3}, $\nu_1 = |\widehat{\mu'_4}(x)|^2 \chi_{[0, 1]}d\lambda $, and $ \nu_2 = \sum_{n \in \mathbb{Z}} |\widehat{\mu'_4}(n)|^2 \delta_n $ are Plancherel measures for $ \mu_4 $, ($ \nu_1, \nu_2  \in \mathcal{F}_{1, 1}(\mu_4) $). For every $ n\in\mathbb{N}$ with $ 1\leq n\leq N $, let $ \rho_n $ be a probability measure. Suppose $ 0 < m\leq \phi(x)\leq M $ on $ \mathbb{R}^d $. Then by Proposition~\ref{SF} we have $ \nu_1, \nu_2  \in \mathcal{F}_{m,M}(\phi\ast\rho_1\ast\cdots\ast\rho_nd\mu_4) $.
\end{example}
\begin{corollary}\label{coo}
\emph{There exist infinitely many absolutely continuous measures which admit discrete and continuous frame measures}.
\end{corollary}
\begin{proof}
Based on Example~\ref{Exp} and Corollary \ref{Cco}, there are infinitely many absolutely continuous measures with respect to $ \mu_4 $ which admit discrete and continuous frame measures. On the other hand, there are also infinitely many absolutely continuous measures with respect to Lebesgue measure which admit discrete and continuous frame measures, since by Theorem \ref{3.9}, there are positive constants $ c $, $ C $ such that for every set $ E \subset \mathbb{R}^d $ of finite Lebesgue measure, a discrete measure $ \nu=\sum_{\lambda \in \Lambda_E} \delta_{\lambda_E} $ is a frame measure for $ \chi_E d\lambda $. Precisely, we have $ \nu  \in \mathcal{F}_{c|E|,C|E|}(\chi_E d\lambda) $. In addition, by Proposition~\ref{3.2}, $ \lambda $ is a Plancherel measure for $ \chi_E d\lambda $ and for any function $ c|E|\leq \phi(x) \leq C|E| $ by Lemma~\ref{123}, $ \phi d\lambda  \in \mathcal{F}_{c|E|,C|E|}(\chi_E d\lambda) $. Besides, if $ E_n \subset \mathbb{R}^d $, $ \lambda(E_n) < \infty $ for $ n\in \{1,\ldots, N \} $ and if $ \mu_n= \frac{1}{\lambda(E_n)}\chi_{E_n} d\lambda $, then by Theorem~\ref{3.4}, we have $ \phi d\lambda, \nu \in \mathcal{F}_{c|E|,C|E|}(\chi_E d(\lambda\ast\mu_1\ast\cdots\mu_n)) $. 
\end{proof}
\begin{lemma}[\cite{6}]
\emph{Let $ \mu $ be a Borel measure on $ \mathbb{R}^d $. Then $ \nu $ is a frame measure for $ \mu $ if and only if $ \nu $ is a frame measure for $ \delta_t \ast \mu $ with the same frame bounds, where $ t\in \mathbb{R}^d $}. 
\end{lemma}
The last lemma from \cite{6} shows that we still can construct infinitely many measures $ \mu $ which admit frame measures $ \nu $.
                                                                                                                                                                    
\section*{Acknowledgements}
The authors would like to thank Dr. Nasser Golestani for his valuable guidance and helpful comments.

\small $^{a}$Department of Mathematics , Science and Research Branch, Islamic Azad University, Tehran, Iran.

\emph{E-mail address}: \small{fz.farhadi61@yahoo.com} \\ 
\emph{E-mail address}: \small{mrmardanbeigi@srbiau.ac.ir}\\ 
 
 $^{b}$Department of Mathematics, Faculty of Science, Islamic Azad University, Central Tehran Branch,Tehran, Iran.
 
\emph{E-mail address}: \small{moh.asgari@iauctb.ac.ir}\\ 


\begin{thebibliography}{10}

\bibitem{1} O. Christensen, An Introduction to Frames and Riesz Bases, Applied and Numerical Harmonic Analysis, Birkh$\ddot{a}$user Boston Inc., Boston, MA, 2003.
\bibitem{2} D. Dutkay, D. Han and E. Weber, Bessel sequence of exponential on fractal measures, J. Funct. Anal. 261 (2011), 2529-2539. 
https://doi.org/10.1016/j.jfa.2011.06.018.
\bibitem{3} D. Dutkay, D. Han and E. Weber, Continuous and discrete Fourier frames for fractal measures, Trans. Amer. Math. Soc. 366 (3) (2014), 1213-1235. https://doi.org/10.1090/S0002-9947-2013-05843-6.
\bibitem{4} D. Dutkay and P. Jorgensen, Fourier frequencies in affine iterated function systems, J. Funct. Anal. 247 (1) (2007), 110-137. 
https://doi.org/10.1016/j.jfa.2007.03.002.
\bibitem{5} D. Dutkay and C.-K. Lai, Uniformity of measures with Fourier frames, Adv. Math. 252 (2014), 684-707. 
https://doi.org/10.1016/j.aim.2013.11.012.
\bibitem{6} X. Fu and C.-K. Lai, Translational absolute continuity and Fourier frames on a sum of singular measures, J. Funct. Anal. 274 (9) (2018), 2477-2498. 
https://doi.org/10.1016/j.jfa.2017.10.021.
\bibitem{7} X.-G. He, C.-K. Lai and K.-S. Lau, Exponential spectra in $L^2(\mu)$, Appl. Comput. Harmon. Anal. 34 (3) (2013), 327-338. 
https://doi.org/10.1016/j.acha.2012.05.003.
\bibitem{8} T.-Y. Hu, K.-S. Lau and X.-Y. Wang, On the absolute continuity of a class of invariant measures, proc. Amer. Math. Soc. 130 (3) (2001), 759-767. 
https://doi.org/10.2307/2699851.
\bibitem{9}  J. E. Hutchinson, Fractals and self-similarity, Indiana Univ. Math. J. 30 (5) (1981), 713-747. https://doi.org/10.1512/iumj.1981.30.30055. 
\bibitem{10} P. Jorgensen and S. Pedersen, Dense analytic subspaces in fractal $L^2$-spaces, J. Anal. Math. 75 (1998), 185-228. 
https://doi.org/10.1007/BF02788699.
\bibitem{11} I. Laba and Y. Wang, On spectral Cantor measures, J. Funct. Anal. 193 (2002), 409-420. https://doi.org/10.1006/jfan.2001.3941.
\bibitem{12} I. Laba and Y. Wang, Some properties of spectral measures, Appl. Comput. Harmon. Anal. 20 (1) (2006), 149-157. 
%https://doi.org/10.1016/j.acha.2005.03.003.
\bibitem{13} C.-K. Lai, On Fourier frame of absolutely continuous measures, J. Funct. Anal. 261 (10) (2011), 2877-2889. 
https://doi.org/10.1016/j.jfa.2011.07.014.
\bibitem{14} C.-K. Lai and Y. Wang, Non-spectral fractal measures with Fourier frames, J. Fractal Geom, 4 (3) (2017), 305-327.
\bibitem{15} N. Lev, Fourier frames for singular measures and pure type phenomena, proc. Amer. Math. Soc. 146 (2018), 2883-2896. 
https://doi.org/10.1090/proc/13849.
\bibitem{16} S. Nitzan, A. Olevskii and A. Ulanovskii, Exponential frames on unbounded sets, Proc. Amer. Math. Soc. 144 (1) (2016), 109-118. 
https://doi.org/10.1090/proc/12868.
\bibitem{17} J. Ortega-Cerd\`a and K. Seip, Fourier frames, Ann. of Math. 155 (3) (2002), 789-806. https://doi.org/10.2307/3062132.
\bibitem{18} G. Picioroaga and E. Weber, Fourier frames for the Cantor-4 set, J. Fourier Anal. Appl. 23 (2) (2017), 324-343. 
https://doi.org/10.1007/s00041-016-09471-0.
\bibitem{19} F. Zeinal Zadeh Farhadi, M. S. Asgari, M. R. Mardanbeigi, M. Azhini, Generalized Bessel and frame measures, Preprint (2018). arXiv:1902.06434.
\end{thebibliography}
\end{document}